\documentclass[12pt]{amsart}

\usepackage{bbm}
\usepackage{amsmath}
\usepackage{cancel}
\usepackage{amsfonts}
\usepackage{amssymb}
\usepackage{amsthm}
\usepackage{amscd}
\usepackage{latexsym}
\usepackage{multirow}
\usepackage{float}
\usepackage{threeparttable}
\usepackage{enumerate}
\usepackage{rotating}
\usepackage{tabularx}
\usepackage[colorinlistoftodos]{todonotes}
\usepackage{hyperref}
\usepackage{comment}
\theoremstyle{plain}
\newtheorem{theorem}{Theorem}[section]
\newtheorem{lemma}[theorem]{Lemma}
\newtheorem{cor}[theorem]{Corollary}

\newtheorem{conj}[theorem]{Conjecture}

\newtheorem{proposition}[theorem]{Proposition}
\newtheorem{defi}[theorem]{Definition}
\newtheorem*{theorem*}{Theorem}
\theoremstyle{remark}
\theoremstyle{remark} 
\theoremstyle{remark} \newtheorem{example}{Example}

\newcommand{\mcal}{\mathcal}

\newcommand{\mfk}{\mathfrak}
\newcommand{\cF}{\mathcal{F}}
\newcommand{\T}{\mathcal{T}}

\newcommand{\stars}{\mfk{S}} %change this to be more specific
\newcommand{\blocked}{\mfk{D}} %change this to be more specific

\newcommand\mysqueeze[2]{%
\newdimen\origspacing%
\newdimen\newspacing%
\origspacing=\fontdimen2\font%
\setlength{\newspacing}{1\origspacing}%
\addtolength{\newspacing}{-#1}%
\fontdimen2\font=\newspacing%
{#2}%
\fontdimen2\font=\origspacing}

\title[Intersecting Families of Trees]{Intersecting Families of Spanning Trees}
\date{}

\author[Frankl]{Peter Frankl}
\address{R\'enyi Institute, Budapest, Hungary}
\email{peter.frankl@gmail.com}

\author[Hurlbert]{Glenn Hurlbert}
\address{Department of Mathematics and Applied Mathematics, Virginia Commonwealth University, Richmond, VA, 23220 USA}
\email{ghurlbert@vcu.edu}

\author[Ihringer]{Ferdinand Ihringer}
\address{Department of Mathematics, Southern University of Science and Technology, No. 1088 Xueyuan Blvd, Nanshan District, Shenzhen, 518055, Guangdong Province, P.R. China.}
\email[Corresponding author]{Ferdinand.Ihringer@gmail.com}

\author[Kupavskii]{Andrey Kupavskii${^{*}}$}
\address{CombGeo lab, Moscow Institute of Physics and Technology, Dolgoprudniy, Russia}
\email{kupavskii@ya.ru}
\thanks{${^*}$Research was supported by the Foundation for the Advancement of Theoretical
Physics and Mathematics ?BASIS?}
    
\author[Lindzey]{Nathan Lindzey}
\address{Department of Mathematical Sciences, University of Memphis, Memphis, TN, 38111 USA}
\email{nathan.lindzey@memphis.edu}

\author[Meagher]{Karen Meagher${^{**}}$ }
\address{Department of Mathematics and Statistics, University of Regina, Regina, SK, S4S 0A2, Canada}
\email{karen.meagher@uregina.ca}
\thanks{${^{**}}$Research supported in part by an NSERC Discovery Research Grant,
    Application No.: RGPIN-2018-03952.}

\author[Pantangi]{Venkata Raghu Tej Pantangi}
\email{pvrt1990@gmail.com}

\begin{document}

\begin{abstract}
A family $\mathcal{F}$ of spanning trees of the complete graph on $n$ vertices $K_n$ is \emph{$t$-intersecting} if any two members have a forest on $t$ edges in common. We prove an Erd\H{o}s--Ko--Rado result for $t$-intersecting families of spanning trees of $K_n$. In particular, we show there exists a constant $C > 0$ such that for all $n \geq C (\log n) t$ the largest $t$-intersecting families are the families consisting of all trees that contain a fixed set of $t$ disjoint edges (as well as the stars on $n$ vertices for $t = 1$). The proof uses the spread approximation technique in conjunction with the Lopsided Lov\'asz Local Lemma. 
\end{abstract}


% \msc{Primary 05C35; Secondary 05C69, 20B05}

\maketitle

\section{Introduction}

A \emph{spanning tree} $T$ of a simple graph $G = (V,E)$ is a connected subgraph $T \subseteq G$ on $n-1$ edges that contains all of $V$. We recall a classic result of Nash-Williams that determines the number of edge-disjoint spanning trees of a graph.
\begin{theorem}[Nash-Williams]
A graph $G = (V,E)$ has $\ell$ edge-disjoint spanning trees if and only if for every partition $V = V_1 \sqcup \cdots \sqcup V_k$ such that $V_i \neq \emptyset$, there are at least $\ell(k-1)$ cross edges of $G$, i.e., edges $uv \in E$ such that $u$ and $v$ do not belong to the same partition class.
\end{theorem}
\noindent Let $\mathcal{T}(G)$ be the collection of all spanning trees of a graph $G$. We define \emph{the spanning tree disjointness graph} $\Gamma(G)$ on $\mathcal{T}(G)$ such that two vertices $T,T' \in \mathcal{T}(G)$ are adjacent if they have no edges in common. 

Let $K_n$ be the complete graph on $n$ vertices. Its spanning tree disjointness graph $\Gamma(K_n)$ has $n^{n-2}$ vertices by Cayley's theorem, and Nash-Williams' theorem implies that there are $\ell = \lfloor n/2 \rfloor$ edge-disjoint spanning trees of $K_n$ (which is best possible). In other words, its \emph{clique number} $\omega(\Gamma(G))$ equals $\lfloor n/2 \rfloor$.  Moreover, using Nash-Williams' theorem, one can determine $\omega(\Gamma(G))$ for arbitrary $G$ in polynomial time. Closed formulas for the clique numbers of $\Gamma(G)$ have been computed for special graph classes (see~\cite{Palmer01}, for example).

In this work, we consider the natural complementary problem of determining the \emph{independence number} $\alpha(\Gamma(G))$ of the spanning tree disjointness graph. The independent sets of the spanning tree disjointness graph of $G$ are families of spanning trees of $G$ such that any two trees \emph{intersect}, i.e., have an edge in common. Independent sets of this graph are called \emph{intersecting families}. We will focus our efforts on $G = K_n$, leaving more general classes of graph as an open problem.

For many structured classes of graphs, there are well-known relations between their clique and independence numbers, e.g., \emph{the clique-coclique bound}~(see~\cite{GMbook}, for example).
\begin{theorem}
    If $\Gamma$ is a vertex-transitive graph on $n$ vertices, then $$\omega(\Gamma)\alpha(\Gamma) \leq n.$$
\end{theorem}
\noindent For even $n$, the theorem above would give us
\[
    n/2 \cdot \alpha(\Gamma(K_n)) \leq n^{n-2}.
\]
\noindent It is a simple exercise to show that the size of a family $\mathcal{F}$ of spanning trees of $K_n$ that all share a fixed edge is $2n^{n-3}$, which would seem to imply that $\mathcal{F}$ is a maximum independent set of $\Gamma(K_n)$.

Of course, the graph $\Gamma(K_n)$ is \emph{not} vertex-transitive. Indeed, it is highly irregular: the orbits of the natural action of $S_n$ on its vertices are the isomorphism classes of $n$-vertex trees. Moreover, one can show that $\mathcal{F}$ unioned with the set of all \emph{$n$-stars} of $K_n$, i.e., spanning trees with one vertex of degree $n-1$, in fact produces  a larger independent set of size $2n^{n-3} + n-2$ (which turns out to be maximum). 

The irregularity of the spanning tree disjointness graph is the source of much difficulty. Indeed, many other natural families of disjointness graphs for other combinatorial domains enjoy much more combinatorial regularity, making them amenable to elegant algebraic and combinatorial methods for computing their independence numbers (e.g., see~\cite{GMbook}). This collection of results and techniques is sometimes referred to as \emph{Erd\H{o}s--Ko--Rado combinatorics}, which we briefly describe below.

\subsection{Erd\H{o}s--Ko--Rado Combinatorics}
Given a positive integer $m$, we set $[m]:=\{1,\ 2,\ \ldots,\ m\}$. Given $1\leq k\leq m$, let $\binom{[m]}{k}$ denote the set of $k$-subsets of $[m]$.
We say that a family of subsets of $[m]$ is {\em $t$-intersecting} if any two of its elements intersect
in at least $t$ elements.

The Erd\H{o}s--Ko--Rado (EKR) Theorem \cite{EKR1961} is one of the most celebrated results in extremal combinatorics
and states that, for $m \geq 2k$, a $1$-intersecting family $\cF$ in $\binom{[m]}{k}$ has size at most $\binom{m-1}{k-1}$
with equality for $m>2k$ if and only if $\cF$ consists of all $k$-sets that contain a fixed element of $[m]$.
This result determines the size and structure of the largest intersecting families of $k$-subsets of $[m]$. Many generalizations of the EKR theorem have been proved for other classes of objects possessing a natural notion of intersection (see~\cite{DF1977, FW86}, for example, or~\cite{MR4421401,GMbook} and the references therein). Before we state our main results, we discuss some previous work that is somewhat related to our tree intersection problem. 

In 1978 Simonovits and S\'{o}s introduced the following problem \cite{SiSo1978A,SiSo1978B}:
for a given family of graphs $L$, what is $f(n, L)$, the largest number
of graphs on $n$ vertices such that the intersection of every pair is in $L$. They show the following results:
\begin{itemize}
    \item if $L = \{ \text{$k$-vertex graphs}\}$, then $f(n, L) = o(n^{k+2})$;
    \item if $L = \{ \text{complete graphs}\}$, then $f(n, L) \geq 2^{n-2}$;
    \item if $L = \{ \text{stars}\}$, then $f(n, L) = 2^{n-1}$;
    \vspace{-0.3em}
    \item if $L = \{ \text{connected graphs}\}$, then $f(n, L) \geq 2^{\binom{n-1}{2}}$;
    \item if $L = \{ \text{non-empty paths}\}$, then $f(n, L) = o(n^4)$; and
    \item if $L = \{ \text{non-empty cycles}\}$, then $f(n, L) = \binom{n} {2}-2$.
\end{itemize}
In 2012, Ellis, Filmus, and Friedgut~\cite{EFF2012} determined the largest family of triangle-intersecting ($K_3$-intersecting) family of graphs, resolving a conjecture by Simonovits and S\'{o}s from 1976.
\begin{theorem}[Ellis, Filmus, and Friedgut~\cite{EFF2012}]
 Let $\cF$ be a triangle-intersecting family of graphs on $n$ vertices. Then
 $|\cF| \leq \frac{1}{8} 2^{\binom{n}{2}}$. Equality holds if and only if $\cF$
 consists of all graphs containing a fixed triangle.
\end{theorem}

\subsection{Main Results}
The main result of this work is an EKR theorem for $n$-vertex spanning trees (\emph{trees}, for short), which in some sense lies in between the intersection problems mentioned above. 
Let $\T_{n} := \mathcal{T}(K_n)$ denote the set of labelled spanning trees of $K_n$, which we sometimes refer to as the \emph{ambient family}. We say that $A, B\in \T_{n}$ are \emph{$t$-intersecting} if they share at least $t$ edges, and we often simply use \emph{intersecting} in place of $1$-intersecting. Recall that an $n$-star is a tree that is isomorphic to the complete bipartite graph $K_{1,n}$. Every tree is connected, so any star must intersect with any other tree. A set of trees is $t$-intersecting (or intersecting) if the trees in the set are pairwise $t$-intersecting (or intersecting). 

We answer the following question for sufficiently large $n$ with respect to $t$: \emph{what is the size and structure of the largest $t$-intersecting families of $\T_{n}$?} This amounts to characterizing the maximum independent sets of the \emph{spanning tree $t$-disjointness graph} $\Gamma_t(K_n)$ defined such that $T,T' \in \mathcal{T}_n$ are not 
adjacent if they have $t$ or more edges in common. It also modifies
the problem by Simonovits and S\'{o}s discussed above by replacing ``number of graphs'' with ``number of spanning trees'',
and taking $L$ to be the set of all forests with at least $t$ edges.

We use the \emph{spread approximation technique} developed in \cite{KZ2024}
to prove the following theorem.\footnote{No attempt was made to optimize the constant $4032$ or the constant $2^{19}$ in the hypotheses of the theorem below.}

\begin{theorem}\label{theorem:main2}
 Suppose that $n \geq 2^{19}$ and $1 < t \leq \frac{n}{4032 \log_2 n}$.
 Let $\cF$ be a $t$-intersecting family of trees on $n$ vertices.
 Then
 \[
  |\cF| \leq 2^t n^{n-t-2}.
 \]
 Moreover, equality holds if and only if $\cF$ consists of all trees
 containing a fixed set of $t$ pairwise disjoint edges.
\end{theorem}
\noindent The stars cause a slight difference between the $t=1$ case and the $t\geq2$ case, so we state the result for the $t=1$ case separately.
\begin{theorem}\label{cor:maint1}
  For $n\ge 2^{19}$, let $\cF$ be a $1$-intersecting family of trees on $n$ vertices.
  Then
  \[
   |\cF| \leq 2n^{n-3} + (n-2).
  \]
  Moreover, equality holds if and only if $\cF$ consists of 
  all stars plus all trees containing one fixed edge.
 \end{theorem}
Apart from the spread approximation technique, the main ingredient of the proof is a bound on the number of trees that contain a certain number of fixed edges, but avoid a given tree (see Section~\ref{sec:BoundBlocked}).
For this, we use probabilistic techniques, most importantly, the \emph{Lopsided Lov\'asz Local Lemma} (LLLL) introduced by Erd\H{o}s and Spencer~\cite{ErdosS91}.

\section{Notation and Background}

Let $2^{[m]}$ denote the power set of $[m]$.
We start by identifying elements of $\T_n$ as sets. We can represent a tree $T \in \T_{n}$ with its edge set, since an edge is an element of $\binom{[n]}{2}$. That is, each tree can be identified with an element of $2^{\binom{[n]}{2}}$. In fact, we can identify every graph on $n$ vertices as a subset of ${\binom{[n]}{2}}$. If the graph is a tree, then this subset will have size exactly $n-1$; if it is a forest with $m$ components, it will have size exactly $n-m$. Conversely, any element in $F \in 2^{\binom{[n]}{2}}$ can be viewed as the set of edges in a graph on $n$ vertices. If this graph is a forest, we say $F$ \textsl{spans a forest} and this forest is \textsl{spanned by $F$}.
The size of a forest, denoted by $|F|$, is the number of edges in the forest, so this representation preserves the size of a forest. The definition of $t$-intersecting trees is also consistent with this identification, trees are $t$-intersecting if and only if their corresponding (edge) sets are $t$-intersecting.

% count trees
Cayley's theorem is a famous result that proves there are exactly $n^{n-2}$ labelled trees on $n$ vertices. 
We will use a generalization of this from~\cite[Lemma 6]{MR3066891} that counts the number of labelled trees that contain a specific subforest.

\begin{lemma}[Lu, Mohr, Sz\'ekely]
\label{lem:treeswithforest}
Let $F$ be a forest in the complete graph on $n$ vertices that has $m$ components of sizes $q_1, q_2, \dots, q_m$. Then the number of spanning trees in $K_n$, that contain $F$ is
\[
q_1 q_2 \cdots q_m \ n^{n-2-\sum_{i=1}^m (q_i-1)} .
\]
\end{lemma}

A subset of $\T_n$ in which all trees contain a common edge is clearly intersecting. Similarly, for any forest $F$ with $t$ edges the set of all trees in $\T_n$ that contain $F$ is $t$-intersecting; we called such sets \textsl{trivially $t$-intersecting}. We will define notation that will be used to describe the set of all trees that contain a common forest.

Given $\mcal{F} \subset 2^{[m]}$, $\mcal{S} \subset 2^{[m]}$, and $X \subset [m]$, we define
\begin{subequations}
\begin{align*}
\mcal{F}(X) :&=\{ A\setminus X  :\    A \in \mcal{F}, \,  X \subset A \}, \\
\mcal{F}[X] :&=\{ A  : \ A\in \mcal{F}, \,  X \subset A \}, \\
\mcal{F}[ \mcal{S} ] :&= \bigcup\limits_{X \in \mcal{S} }\mcal{F}[X].
\end{align*}
\end{subequations}

If $F \in 2^{\binom{[n]}{2}}$ spans a forest, then the set $\T_n[F]$ is the set of all trees in $\T_n$ whose edges include the elements in $F$; this set is trivially $|F|$-intersecting. The size of $\T_n[F]$ can be easily found using Lemma~\ref{lem:treeswithforest}. We note one particular case that we consider later.

\begin{lemma}\label{lem:disjointedges}
If $F$ is forest with exactly $\ell$ disjoint edges then
\[
|\T_n[F]|=2^\ell n^{n-2-\ell}.
\]
Moreover, if $F'$ is any other forest with $\ell$ edges then $| \T_n[F'] | \leq  |\T_n[F]| $.
\end{lemma}
\begin{proof}
The size of $\T_n[F]$ clearly follows from Lemma~\ref{lem:treeswithforest}.

Let $F'$ be any other forest with $\ell$ edges and components of sizes $p_1, \dots, p_{k}$.
Set $p_{k+1} =1 ,\dots,p_\ell=1$, this simply adds isolated vertices to $F'$ and does not change the number of trees that contain the forest. Then $\sum_{i=1}^\ell p_i = 2 \ell$, and using the AM/GM inequality, the maximum value of $p_1 \cdots p_\ell$ is achieved with $\ell$ of the values equal to 2.
\end{proof}

We also give a simple lower bound on the number of trees that contain a fixed forest.

\begin{cor}\label{cor:lowerbound}
Let $F$ be a forest with $t$ edges in the complete graph on $n$ vertices, then the number of spanning trees in $K_n$, that contain $F$ is at least $n^{n-t-2}$.
\end{cor}
\begin{proof}
Let $q_1, q_2, \dots, q_m$ be the number of vertices in each of the components in $F$.
From Lemma~\ref{lem:treeswithforest}, along with the fact that each $q_1 \geq 2$, and $\sum_{i=1}^m q_i = t+m$,
it follows that
\[
|\T_n[F]|  = q_1 q_2 \cdots q_m \ n^{n-2-\sum_{i=1}^m (q_i-1)}  \geq n^{n-2-t}.\qedhere\]
\end{proof}

%%%%%%%%%%%%%%%%%%%%%%%%%%%%%%
If two trees do not intersect, then we say these trees \textsl{avoid} each other; this happens if the two trees have no edges in common.
Recall that a \textsl{star} is a tree that is isomorphic to $K_{1,n}$. Let $\stars \subset \T_n$ be the set of all stars in $\T_n$. Note that $| \stars |=n$. Since any tree $T \in \T_n$ is connected, no $T$ avoids any of the stars in $\stars$. 

For a given tree $T_0 \not \in \stars$, we will consider the set of trees that contain some fixed forest $F \nsubseteq T_0$ and avoid $T_0$, outside of $F$.
We denote this set as
\[
\mathcal{T}_n[T_0;F] :=\{ T \in \T_n[F] \ : \  (T \cap T_0) \setminus F = \emptyset \}.
\]
We are interested in the following minimization over all forests $F$ of size $t$:
\[
\blocked_t:= \min_{\substack{F \text{ forest}\\|F|=t}} \hspace*{0.6cm} \min_{\substack{T_0 \in \T_n \setminus \stars \\ |T_0 \cap F| < t}} \hspace*{0.3cm}  |  \mathcal{T}_n[T_0;F]  |.
\]
 The value of $\blocked_t$ is important for EKR theorems.
 If $\mathcal{F}$ is a $t$-intersecting set of trees such that $T_0 \in \mathcal{F}$, but $T_0$ does not contain the forest $F$, then at least $\blocked_t$ of the elements from $\T_n[F]$ do not belong to $\mathcal{F}$, as they do not $t$-intersect with $T_0$.

%-----------------------------------------------------------------------------------------------------------
%
\section{A Lower Bound on \texorpdfstring{$\blocked_t$}{Dt}}
\label{sec:BoundBlocked}

The main result of this section is the following lower bound on $\blocked_t$.

%changed the 2t to 3t
 \begin{proposition}\label{lem:rsize-tintersecting}
If $n \geq  2t+110$, then $\blocked_t > n^{n-2t-17} $.
\end{proposition}

We will apply the Lopsided Lov\'asz Local Lemma (Lemma~\ref{lem:LLLL}) to prove Proposition \ref{lem:rsize-tintersecting}, so first we need to set our notation.

Let $A_1,A_2,\ldots,\allowbreak A_N$ be events. A \emph{negative dependency graph} is a simple graph $G = ([N],E)$ such that
\[
    \Pr \left [\bigwedge_{j \in S} \bar{A}_j \right ] \neq 0 \Rightarrow \Pr\left [A_i~\bigg |\bigwedge_{j \in S} \bar{A}_j \right ] \leq \Pr[A_i]
\]
for all  $i \in [N]$ and $S \subseteq  \{ j \in [N] : \{i,j\} \notin E\}$.

\begin{lemma}[Lopsided Lov\'asz Local Lemma~\cite{ErdosS91} (LLLL)]\label{lem:LLLL}
    Let $A_1,A_2,\allowbreak\ldots, A_N$ be events with negative dependency graph $G = ([N],E)$. If $x_1,x_2,\allowbreak \ldots,x_N \in [0,1)$ and
    \[
        \Pr[A_i] \leq x_i \prod_{ij \in E} (1-x_j) \quad \text{ for all } i \in [N],
    \]
    then
    \[
        \Pr\left [\bigwedge_{i=1}^N \bar{A}_i \right ] \geq \prod_{i=1}^N (1-x_i) > 0.
    \]
\end{lemma}

\noindent We will use the following result of Lu, Mohr, and Sz\'ekely~\cite{MR3066891}. 

\begin{theorem}\label{thm:LMS}\cite[Theorem 4]{MR3066891}
 Let $\mathcal{F}$ be a family of forests of $K_n$.
 The graph of events $\{ A_F: F \in \mathcal{F} \}$, where $A_F \sim A_{F'}$
 if there exist connected components of $C$ of $F$ and $C'$ of $F'$
 that are neither identical nor disjoint, is a negative dependency graph.
\end{theorem}

Let $F$ be a forest with $t$ edges in $K_n$, and $e$ be an edge of $F$. Define $A_e$ to be the event that the edge $e$ is in a spanning tree of $K_n$ drawn uniformly at random. Recall that the \emph{line graph} $\mathcal{L}(G)$ of a graph $G = (V,E)$ is the graph with vertex set $E$ where $e, e' \in E$ are adjacent if they share an endpoint.
Theorem~\ref{thm:LMS} implies the following.

\begin{cor}\label{cor:negdepGraph}
    Let $T \in \mathcal{T}_n$.
    The line graph $\mathcal{L}(T)$ is a negative dependency graph for the events $\{A_e\}_{e \in T}$.
\end{cor}

%-----------------------------------------------------------------------------------------------------------
%main

It follows directly from Lemma~\ref{lem:treeswithforest} that $\Pr[A_e] = \frac2n$.

We say that a forest $T$ is \emph{$d$-star-like} if it has an edge that is incident to at least $(n-1)/d$ other edges of $T$. If a forest is $d$-star-like, then its line graph will have a vertex with degree at least $(n-1)/d$.

The following lemma shows that if $T_0$ is not \emph{$d$-star-like}, then the value of $\mathcal{T}_n[T_0;\emptyset]$ must be large.

%%KAREN big change here!!
%% what was with the $t$???
\begin{lemma}\label{lem:notstar}
 Let $n \geq 5$. % Let $F$ be a forest with $t$ edges in $K_n$.
 Suppose that $T_0$ is a forest of $K_n$ that is not $6$-star-like.
 Then at least $e^{-4} n^{n-2}$ of the trees $T \in \mathcal{T}_n$ avoid $T_0$, so in particular
 \[
 |  \mathcal{T}_n[T_0; \emptyset ]  | \geq e^{-4} n^{n-2}.
 \]
\end{lemma}
\begin{proof}
Let $\{e_1,e_2, \ldots, e_m\}$ be the edge set of $T_0$; this means $m \leq n-1$.  Let
$A_{e_i}$ be the event that $e_i$ is in a spanning tree from
$\mathcal{T}_n$
drawn uniformly at random.  First, we show that the hypotheses of Lemma \ref{lem:LLLL} hold for the events $A_{e_1}, \dots, A_{e_m}$.

Define
\[
    x_{e_1} = x_{e_2} =  \cdots = x_{e_{m}} = 4/n,
\]
and recall that $\Pr[A_e] \leq   \frac{2}{n} $. %  from Lemma~\ref{lem:probA_e}
From Corollary,~\ref{cor:negdepGraph}, the line graph $\mathcal{L}(T_0)$ is the negative dependency graph for $A_{e_1}, \dots, A_{e_m}$. If we set $E= E(\mathcal{L}(T_0))$, then we must show for each $e_i \in T_0$ that
\begin{equation}\label{eq:upperbound}
    \Pr[A_{e_i}] \leq x_{e_i} \prod_{ \{e_i,e_j\}  \in E} (1-x_{e_j}).
\end{equation}
The assumption that $T_0$ is not $6$-star-like guarantees that the maximum degree of $\mathcal{L}(T_0)$ is less than $n/6$, thus
\begin{equation*}\label{eq:secondbound}
       \frac{4}{n}\left( 1 - \frac{4}{n} \right)^{n/6}
       \leq
       x_{e_i} \prod_{ \{e_i,e_j\}  \in E} (1-x_{e_j}).
\end{equation*}
But, as $\Pr[A_{e_i}] \leq  \frac{2}{n}$, it is sufficient to show that
\begin{equation*}
         \frac{2}{n} \leq \frac{4}{n}\left( 1 - \frac{4}{n} \right)^{n/6}.
\end{equation*}
Multiplying by $n$ gives
\[
       2 \leq 4 \left( 1 - \frac{4}{n} \right)^{n/6}.
\]
One can check that this inequality holds for $n = 5$. Moreover, $ \left( 1 - \frac{4}{x} \right)^{x}$ is an increasing function of $x$ for $x\ge 4$, and so the displayed inequality holds, provided that $n \geq 5$. Thus, Equation (\ref{eq:upperbound}) holds for all $n\ge 5.$

Applying LLLL, we deduce that the probability that a spanning tree in $\mathcal{T}_n$, drawn uniformly at random, contains none of the edges $e_1, \dots, e_m$ is
\[
 \Pr \left [ \bigwedge_{i=1}^{m} \bar{A}_{e_i} \right ]
        \geq \prod_{i=1}^{m} \left (1 - x_i \right )
        = \left (1 - \frac{4}{n} \right )^{m}
       \geq e^{-4}.
\]
Corollary~\ref{cor:lowerbound} now shows the number of trees that avoid $T_0$    is
\[
    |  \mathcal{T}_n[T_0; \emptyset ]  |
 \geq e^{-4} \, \mathcal{T}_n
 \geq e^{-4} n^{n-2}. \qedhere
\]
\end{proof}
Before we begin the proof of Proposition \ref{lem:rsize-tintersecting}, we first establish some basic graph-theoretical notation. Let $T$ be a forest.  If $v$ is a vertex of $T$, then the operation $T \setminus \{v\}$ removes the vertex $v$ from $T$ and removes all edges of $T$ that are incident to $v$. If $E$ is a set of edges, then the operation $T \setminus E$ removes from $T$ the edges in $E$ that are contained in $T$. If $F$ is a forest, then the operation $T \setminus F$ removes the edge set of $F$ from $T$.

%%%%%%%%%%%%%%%%%%%%%%%%%%%%%
\begin{proof}[Proof of Proposition \ref{lem:rsize-tintersecting}]
Assume $n \geq 2t+110$. We show for any tree $T_0$ and forest $F_0$ with $t$ edges, that
\[
\mathcal{T}_n [T_0 ; F_0] \geq n^{n-2t-17}.
\]
If $T_0$ is not $6$-star-like and $F_0$ is empty,
then the proof is immediate from Lemma~\ref{lem:notstar}, so we may assume that $F_0$ is non-empty or that $T_0$ is $6$-star-like.

We will construct a finite sequence of pairs of forests $T_i, F_i$ so that for each $i$
we have the lower bound $\mathcal{T}_n [T_i ; F_i] \geq \mathcal{T}_n [T_{i+1} ; F_{i+1}] $. The last pair $T_s,F_s$ that we construct will have the property that $F_{s}$ is empty and $T_{s}$ is not $6$-star like. Our plan has two stages:
\begin{enumerate}
 \item first, remove vertices until $F_s$ is empty;
 \item second, remove all high degree vertices from $T_s$ until it is not $6$-star like.
\end{enumerate}

In the first stage we will delete vertices from $T_0$, and so also from $F_0$, until there are no longer any edges in $F_s$.
We assume $F_0$ is non-empty, otherwise we move to the second stage, in particular, we assume $F_0$ has $t>0$ edges. At this stage, we remove at most $t$ vertices. We will do this carefully, one vertex at a time in a way that ensures the resulting $T_i$'s are never stars. To this end, let $w$ be a vertex that is not adjacent to a designated vertex $w^*$ of highest degree in $T_0$; since $T_0$ is not a star, such a vertex exists.

Let $v_0$ be a non-isolated vertex of $F_0$ with $v_0 \neq w,w^*$. We will delete $v_0$ and replace $F_0$
by a forest $F_1$ of $K_n \setminus \{ v_0 \}$ with fewer edges than $F_0$, and $T_0$ by a forest $T_1$ so that $| \mathcal{T}_n[T_0 ; F_0] | \geq | \mathcal{T}_{n-1}[T_1 ; F_1] |$.

Consider the neighborhood of $v_0$ in $F_0$. We count the number of trees $T$ that contain $F_0$, avoid $T_0 \setminus F_0$, and have the additional property that the neighborhood of $v_0$ in $T$ is exactly the neighborhood of $v_0$ in $F_0$. Since in every such tree $T$ the neighborhood of $v_0$ is fixed, we may replace $T_0$ with $T_1 = T_0 \setminus \{v_0\}$ so that this set of trees is precisely those trees $T \in \mathcal{T}_n[T_0 \setminus \{v_0\} ; F_0]$ in which every neighbor of $v_0$ in $T$ is also a neighbor of $v_0$ in $F_0$. We denote this set by $\mathcal{T}^\ast_n[T_0 \setminus \{v_0\} ; F_0]$.  We will obtain a lower bound on $|\mathcal{T}^\ast_n[T_0 \setminus \{v_0\}; F_0]|$ by considering another related set of trees.

Set $E_0 = \left\{ \{v_0,x_1\}, \{v_0,x_2\}, \ldots, \{v_0, x_m\} \right \}$ to be all the edges of $F_0$ that contain $v_0$. Since $v_0$ is not an isolated vertex, $E_0 \neq \emptyset$.
We define the path
$$E_0' = \{ \{x_1,x_2\}, \{x_2,x_3\}, \ldots, \{x_{m-1},x_m\} \};$$
if $|E_0| = 1$, then $E_0'$ is the empty path.
No edge of the form $\{x_i, x_j\}$ belongs to $F_0$; otherwise, the vertices $\{ v_0, x_i, x_j\}$
would form a cycle in $F_0$. We define
$
F_1 = (F_0 \setminus \{ v_0 \}) \cup E_0'
$
so that $F_1$ is a forest in $K_n \setminus \{ v_0 \}$, and define $T_1 =  ( T_0 \setminus \{v_0 \} ) \setminus E_0$
so that $T_1$ is a forest. Note that $|F_1|\le |F_0|-1.$
We claim that
\[
| \mathcal{T}_{n-1}[T_1 ; F_1] | \leq | \mathcal{T}_n^\ast [ T_0 \setminus\{v_0\}; F_0] |,
\]
which in turn implies that $| \mathcal{T}_{n-1}[T_1 ; F_1] | \leq | \mathcal{T}_n[T_0; F_0]|$.
To see this claim, observe that every tree $T' \in \mathcal{T}_{n-1}[T_1; F_1]$ can be assigned to a unique tree $(T' \setminus E_0') \cup E_0  \in \mathcal{T}_n^\ast [ T_0 \setminus \{v_0\} ; F_0]$.

After we repeat this at most $s \leq t$ times, we have $F_s = \emptyset$. For the sake of simplicity, assume $s=t$.  

We claim that $T_t \in \mathcal{T}_{n-t}$ is not a $(n-t)$-star. To see this, first recall that we never deleted $w$ or $w^*$, so we have $w,w^* \in V(T_t)$. Recall that $w^*$ is a vertex of highest degree in $T_0$. For a contradiction, suppose that $T_t$ is a $(n-t)$-star. Note that $w^*$ cannot be the center vertex of $T_t$, since $w \in V(T_t)$ and $w$ is a non-neighbor of $w^*$. Thus $w^*$ has degree 1 in $T_t$. Since at most 1 edge incident to $w^*$ is removed after each step, we deduce that the maximum degree of $T_0$ is at most $t+1$. But then the center of the star $T_s$ has degree $n-t-1$, which is a contradiction since $n \geq 2t + 100$.

%Note that $T_t$ is not a star since we never deleted the non-neighbor $w$ of our designated vertex $w^*$ of highest degree in $T_0$. 

Having established that $T_t$ is a not a star and $F_t$ has no edges, we move to the second stage. If $T_t$ is not $6$-star-like, then by Lemma~\ref{lem:notstar}
$$
 |  \mathcal{T}_{n-t}[T_t; \emptyset ]  | \geq e^{-4} n^{n-t-2}
$$
and we are done. So we may assume that $T_t$ is $6$-star-like. The plan is to remove vertices from $T_t$ until it is no longer 6-star-like, and then we will again apply Lemma~\ref{lem:notstar}. Denote the vertex in $T_t$ of largest degree by $v_t$. Since $T_t$ is $6$-star-like, we have $\deg(v_t) \geq (n-t-1)/12$. 

%As $T_t$ is not a star in $K_{n-t}$, there is an edge $\{ v_t, u_t \}$ not in $T_t$.

Now we set $T_{t+1} = T_t \backslash \{ v_t \}$ and keep $H_{t+1} = H_t$ the empty graph.
Since $T_t$ is a tree and $v_t$ the vertex with the largest degree, $T_{t+1}$ will not be star.  If $T_{t+1}$ is $6$-star-like, then we repeat the procedure above, starting with $T_{t+1}$ in place of $T_t$.  Again, we pick a vertex $v_{t+1}$ in $T_t$ of largest degree, so that $\deg(v_{t+1}) \geq \frac{n-t-2}{12}$, (and we keep $H_{t+i}$ the empty graph). 

We repeat this process until $T_{t+i}$ is not $6$-star-like, we claim that we will to repeat this process at most 12 times, i.e., $i \leq 12$. Indeed, after 12 iterations we obtain the forest $T_{t+12}$ which cannot be $6$-star-like since it would only have at most
\[
( n-t-1 ) - \frac{1}{12} \sum_{i=1}^{11} (n-t-i) = \frac{n-t+54}{12} < \frac{n-t-13}{6}
\]
edges from $T_0$ left; the last inequality requires $n \geq t+82$, which follows as we assume that $n\ge 2t+110$.

Applying Lemma~\ref{lem:notstar} to $K_{n-t-12}$ in place of $K_n$, and $T_{t+12}$ in place of $T_0$ gives the worst case bound of
\[
 \mathcal{T}_{n-12}[T_{t+12}; F_{t+12}] \geq e^{-4} (n-t-12)^{ (n-t-12)-2}.  
 \]
It can be shown that $(n-t-12)^{ n-t-14} > n^{n-2t-16}$ for $n \geq 2t +110$ (to do see this, take the logarithm of both sides of the inequality and use the Taylor expansion of $\log(n-(t+12))$. This, together with the fact $n \geq 55 > e^4$, gives the lower bound
\[
 \mathcal{T}_{n-12}[T_{t+12}; F_{t+12}] \geq n^{n-2t-17}.  
 \]
\end{proof}

\section{Trees have Spread Properties}

In this section, we provide a bound on the spreadness of the family of all spanning trees. Spreadness is a very handy definition that captures the amount of quasi-randomness of families of sets and is the basis for constructing spread approximations of families. % p from the point of view of consider the ``spreadness" of the trees. This property is swell.
 
%%%%%% r-spread
\begin{defi}[$r$-spread]
Given $r>1$, a family $\mcal{F} \subset 2^{[m]}$ is {\em $r$-spread} if 
\[
|\mcal{F}(X)| \leq r^{-|X|} \: |\mcal{F}| , 
\] 
for all $X \subset [m]$.
\end{defi}

\begin{defi}[$(r,t)$-spread]
A family $\mcal{F} \subset 2^{[m]}$ is said to be {\em $(r, t)$-spread}, if for each set $T \subset [m]$, with size at most $t$, the family $\mcal{F}(T)$ is $r$-spread.
That is, for any $T \subset [m]$ with $|T| \leq t$ and any set $U$ with $T \subseteq U \subseteq [m]$ 
\[
| \mcal{F}(U) | \leq r^{- ( |U| - |T| ) } \: | \mcal{F}(T)|.
\]
\end{defi}

\begin{lemma}\label{lem:spread}
The family $\T_n \subset 2^{\binom{[n]}{2}}$ is $(n/2, n-1)$-spread.
\end{lemma}     
\begin{proof}
Let $T$ be any subset from $\binom{[n]}{2}$ of size no more than $n-1$. If the graph corresponding to $T$ is not a forest, then 
$\mcal{T}_n(T) =\emptyset$. So trivially $\mcal{T}_n(T)$ is $n/2$-spread.

Assume the graph corresponding to $T$ is a forest with components of sizes $q_1,q_2,\dots,q_\ell$.
Let $T \subset U \subset \binom{[n]}{2}$ be such that the graph corresponding to $U$ is a forest and $|U| = |T|+1$.
The forest corresponding to $U$ is a forest formed by adding a single edge to the forest corresponding to $T$.
The components in $U$ will have either sizes equal to : $q_1 + q_2, q_3, \dots, q_\ell$;
$q_1 + 1, q_2, \dots, q_\ell$; or $q_1, \dots, q_\ell, 2$. In any case, if the sizes of the components in $U$ are $p_1, \dots, p_k$, then ratio of $q_1 \dots q_\ell$ to $p_1 \dots p_k$ is at least $1/2$. Thus, by Lemma~\ref{lem:treeswithforest}
\[
\dfrac{ \mcal{T}_n(T)}{\mcal{T}_n(U) }
= \dfrac{ q_1 \dots q_\ell n^{n-2 - \sum_{i=1}^\ell {q_i - 1} } }{ p_1 \dots p_{k} n^{n-2 - \sum_{i=1}^\ell {p_i - 1}}  } 
= \dfrac{ n^{n-2 - |T|} }{ 2 n^{n-2 - |U|} } 
\geq \frac{n}{2}. \qedhere
\]
\end{proof}

%%%%%%%%%%%%%%%%%%%%%%%%%%%%%%%%%%%%%%%%
\section{Theorems for Spread Approximations}

We will follow the method of spread approximations, introduced in \cite{KZ2024} and developed in subsequent papers \cite{kupavskii2023erd, kupavskii2023chv}. The first result allows us to get a low-uniformity approximation for a bulk of the family.
Let $\binom{[m]}{\leq k}$ denote the set of subsets $[m]$ of size at most $k$.

\begin{theorem}\label{thm:rqspreadgen}\cite[Theorem 12]{kupavskii2023erd}
 Let integers $m, k, t, q, r, r_0\ge 1$ satisfy $r\ge 2q$ and $r_0 > r > 2^{12}\log_2(2k)$.
        Suppose that $\mathcal{A}\subset 2^{[m]}$ is $r_0$-spread and that $\mathcal{F}\subset\mathcal{A}\cap\binom{[m]}{\le k}$ is $t$-intersecting.
        Then there exists a $t$-intersecting family $\mathcal{S}\subset\binom{[m]}{\le q}$ and an $\mathcal{F'}\subset\mathcal{F}$ such that
        \begin{enumerate}
            \item
            $|\mathcal{F'}|\le (r_0/r)^{-q-1}|\mathcal{A}|$,
            \item
            $\mathcal{F}\setminus\mathcal{F'}\subset\mathcal{A}[\mathcal{S}]$, and
            \item
            $\mathcal{F}(B)$ is $r$-spread for every $B\in\mathcal{S}$.
        \end{enumerate}
\end{theorem}

We say that $\mathcal{A}$ is the \emph{ambient family}, 
$\mathcal{S}$ is the \emph{spread approximation} to $\mathcal{F}$, and $\mathcal{F}'$ is the \emph{remainder}.

Next, we apply another general result from~\cite{kupavskii2023erd}. This result is used after the spread approximation is found, and gives us the desired conclusion that the structure of $\mcal S$ for large $t$-intersecting families is trivial.

\begin{theorem}{\cite[Theorem 14]{kupavskii2023erd}}
\label{thm:trivialbiggergen}
Let $\epsilon \in (0,1]$, and $m, r_0, q, t \geq 1$ be such that 
$\epsilon r_0 \geq 24q$.
Let $\mcal{A} \subset 2^{[m]}$ be an $(r_0, t)$-spread family 
and let $\mcal{S} \subset \binom{m}{\leq q}$ be a non-trivial $t$-intersecting 
family. Then there exists a $t$-element set $T$ such that $| \mcal{A} [\mathcal{S}] | \leq \epsilon | \mcal{A} [T] |$.
\end{theorem}
{\bf Remark. } In the paper \cite{kupavskii2023erd}, it is required that $\mcal A$ is {\it weakly} $(r,t)$-spread, but, as the name suggests, \mysqueeze{0.1pt}{this is implied by its $(r,t)$-spreadness.}

%%%%%%%%%%%%%%%%%%%%%%%%%%%%%%%%%%%%%%%%%%%%%%%%%%%%%%%%
\section{Proof of Theorem \ref{theorem:main2}}

Recall from Lemma~\ref{lem:disjointedges} that the largest trivial $t$-intersecting family of trees is the family of all trees containing a fixed set of $t$ disjoint edges provided that $t \leq n/2$. The size of this family is
\[
2^t n^{n-t-2}.
\]

\begin{proof}[Proof of Theorem \ref{theorem:main2}]
Let $\mcal{F}$ be a $t$-intersecting family of $\T_{n}$. As $\T_{n}$ is $n/2$-spread by Lemma~\ref{lem:spread}, we can apply Theorem~\ref{thm:rqspreadgen},
with the following parameters:
\begin{enumerate}
\item $\mathcal{A} = \mathcal{T}_n$, the set of all labelled trees on $n$ vertices,
\item $m = \binom{n}{2}$, corresponding to the set of all possible edges,
\item $q=42 t \log_2n$,
\item $r_{0}=\dfrac{n}{2}$, $r = \dfrac{n}{4}$, and
\item $k=n-1$.
\end{enumerate}
With these values we have 
$r > 2^{12} \log_2(2k)$
for $n \geq 2^{19}$. We also need that $r \geq 2 q$,
so
\[
\frac{n}{4} \geq 2 ( 42t \log_2{n} ),
\]
which is implied by the assumed bound $t \leq \frac{n}{4032 \log_2{n}}$.

The ambient family $\T_n$ with these parameters satisfies the conditions of Theorem~\ref{thm:rqspreadgen}. Applying this theorem yields that there exists a $t$-intersecting family $\mathcal{S}$ of sets of size at most $q=42t \log_2{n}$  and a remainder $\mathcal{F}' \subset \mathcal{F}$. By Theorem~\ref{thm:rqspreadgen}, the bound on the remainder is
\begin{align*} |\mathcal{F}'|  \leq \left(\frac{r_0}{r}\right)^{-q-1} |\T_n|
                 &\leq 2^{ -42t \log_2{n} -1 } |\T_n| \\
&\leq n^{ -42t } |\T_n| = n^{ -42t } n^{n-2} = n^{n-42t-2}. \end{align*}

Next we show that if the family $\mathcal{F} \setminus \mathcal{F}' \subset \T_{n}[\mathcal{S}]$
is close to extremal, then $\mathcal{S}$ is trivial.
This is done by applying Theorem~\ref{thm:trivialbiggergen} to the family ${\T_n}[\mathcal{S}]$.
By Lemma~\ref{lem:spread} $\mathcal{F}$ is  $(n/2, t)$-spread, so in order to apply Theorem~\ref{thm:trivialbiggergen} we need that $\epsilon r_0  \geq 24 q$. We will use $\epsilon = \frac12$, and $q:=42t\log_2{n}$, so we need that
\[
n \geq 4032 t \log_2{n}.
\]
This is satisfied since we assume that
$t \leq \frac{n}{4032 \log_2{n}}$.

By Theorem~\ref{thm:trivialbiggergen}, if $\mathcal S$ is non-trivially $t$-intersecting, then
\[
| \mathcal{F} \setminus \mathcal{F}' |  \leq \epsilon | \T_n[F] |
\]
for some $t$-element $F$  (i.e., a forest on $t$ edges). Since $| \T_n[F] | \ge n^{n-2-t}$ and $|\mathcal F'|\le n^{n-42t-2}$, we have
\[
|\mathcal F| = | \mathcal{F} \setminus \mathcal{F}' |+|\mathcal F'|  \leq \left(\frac{1}{2}+n^{-41t} \right)| \T_n[F] |<| \T_n[F] |
\]
in the case where $\mathcal S$ is non-trivial.

Thus, we may assume $\mathcal{S}$ is a trivial $t$-intersecting set, i.e., $\mathcal{S} = \{F\}$ where $F$ is a forest on $t$ edges. Without loss of generality, we may assume that $\mathcal F'\cap \T_n[F]=\emptyset.$ If $\mcal{F}'$ is empty, then $\mathcal{F}$ is trivial $t$-intersecting and $|\mathcal{F} | = 2^t n^{n-t-2}$, so we assume that $\mcal{F}'$ is non-empty and pick $T_0 \in \mcal{F}'$.

%KAREN add the n-2t-17
First, we assume that we can pick a $T_0 \in \mcal{F}'$ that is not a star. Proposition~\ref{lem:rsize-tintersecting} states that $n^{n-2t-17}$ is a lower bound on the number of trees on $n$ vertices that contain $F$ and avoid $T_0$, outside $F$.
This means that the number of sets in $\mathcal{T}_n[T_0; F]$, that is, the number of trees that contain $F$ and do not intersect $T_0$ outside $T_0 \cap F$, for any forest $F$ of $t$ edges is much larger than $|\mcal{F}'|$.
So
\begin{align*}
|\mathcal{F}| &\leq |\T_n(F)| - |\mathcal{T}_n[T_0; F]| + |\mathcal{F}'|\\
&\leq 2^tn^{n-t-2} - n^{n-2t-17}  + n^{n-42t-2}
< 2^t n^{n-t-2}.
\end{align*}

Now assume that every $T_0 \in \mcal{F}'$ is a star.
If $t=1$, there is nothing to do, as we have noted, a star intersects any tree and our claim is $\mcal{F} = \mcal{F} \cup \stars$.
Thus, we assume $t \geq 2$.
No two stars are $t$-intersecting, and thus $|\mcal{F}'|= 1$. Let $T_0$ be the unique star in $\mcal{F}'$, and denote its only non-leaf vertex by $v$. If $v$ is adjacent to an edge in $F$, then, since $|F|>1$, there are at least 2 vertices in $F$ that are not $v$. In this case, we can construct at least two trees from $\T_{n}[F]$ that do not intersect $T_0$ outside $F$, by simply taking $F$ and making every vertex not in $F$ incident to any of the vertices in $F$ that are not $v$ (and perhaps adding edges between the components of $F$, without using the vertex $v$). If $v$ is not adjacent to an edge of $F$, then any tree with $F$ in which $v$ is a leaf will not be $t$-intersecting with $T_0$.

 Thus $|\mcal{F} \setminus \mcal{F}'| \leq |\T_{n}[F]|-2$, and we have
 \[
 |\mcal{F} \setminus \mcal{F}'| + |\mcal{F}'| \leq |\T_{n}[F]|-1,
 \]
 as desired. This completes the proof of the main result.
\end{proof}

\section{Open Problems}

In the next section we discuss our main conjectures for a complete result for $t$-intersecting trees.

\subsection{Conjectures for Larger \texorpdfstring{$t$}{t}}

In this section, we discuss, the form we expect the largest $t$-intersecting family of trees to take when $t$ is not bounded away from $n$. For any $t\le n/2$ we conjecture that a trivial $t$-intersecting family of trees is still the largest possible.

\begin{conj}\label{con:weak}
If $t \leq n/2$, then the largest $t$-intersecting family of trees is the family of all trees that contain a fixed set of $t$ disjoint edges.
\end{conj}

If $t>n/2$ then it is not possible to have a set of $t$ disjoint edges. We will consider an example of a $t$-intersecting set that is larger than the trivial $t$-intersecting one.

\begin{example}
Assume $t$ is even and $ \frac{3(t+2)}{2}  \leq n < 2t$, $n\ge 8$. Let $F$ be a forest that consists of  $t/2$ disjoint copies of a path on $3$ vertices. Note that the number of vertices is sufficient to host such a forest. By Lemma~\ref{lem:treeswithforest} the number of trees that contains $F$ is
\[
3^{t/2} n^{n-2-t}.
\]

Since $\frac{3(t+2)}{2} \leq n$, there is also a forest consisting of $t/2 +1$ disjoint copies of a path on $3$ vertices. If $F$ is the set of edges in this forest, then $|F| =t+2$. The collection of all trees that contain at least $t+1$ of these edges is a $t$-intersecting set.
Let us count the number of such trees using a simple inclusion-exclusion argument.

The number of trees that contain all the $t+2$ edges in $F$ is
\[
3^{t/2 + 1} n^{n-2-(t+2)} = 3^{t/2 + 1} n^{n-t-4}. 
\]

Next we count the number of trees that contain at least $t+1$ of the edges in $F$. By Lemma~\ref{lem:treeswithforest}, the number of trees that contain a forest with $(t+2)/2 - 1 = t/2$ components of size 3 and one more component of size 2 is 
\[
2 ( 3^{t/2}) n^{n-2-(t+1))} =2 ( 3^{t/2} ) n^{n-t-3};
\]
this number includes the trees that contain all of $F$.
There are $t+2$ ways to pick the component of size 2 and, adjusting for over-counting the tree that contain all of $F$, the number of such trees that contain at least $t+1$ of the edges in $F$ is
\begin{align*}
&(t+2) (2) 3^{t/2} n^{n-t-3}  - (t+1)  3^{t/2 + 1} n^{n-t-4}\\
= ~& 3^{t/2}   n^{n-t-4} \left( (t+2) 2 n - (t+1) 3 \right) \\
= ~& 3^{t/2}   n^{n-t-4} \left( 2nt + 4n - 3t - 3 \right).
\end{align*}

The second construction is larger if
\[
3^{t/2}   n^{n-t-4} \left( 2nt + 4n - 3t - 3 \right) > 3^{t/2} n^{n-2-t},
\]
which happens if and only if
\[
 n^2 -(4+2t)n + 3t+3 < 0. 
\]
Since this parabola is negative for $n=2$ and for $n = 2t$, it is negative for all $n \in [3(t+2)/2,\dots, 2t]$.
So the second construction is always larger.
\end{example}

We end with a ``Complete Intersection Theorem''-type conjecture.
Define $F_{n, \ell}$ to be any spanning forest on $n$ vertices with $\ell$ edges in which each component has either $k$ or $k+1$ vertices (so the size of the components are as equal as possible). Then let $\mathcal{F}_{n,t,j}$ be the family of all trees that contain at least $t+j$ of the $t+2j$ edges of $F_{n,t+2j}$.

\begin{conj}\label{con:AK}
For any $t$ and $n$, there exists a $j$ so that $\mathcal{F}_{n,t,j}$ is the largest $t$-intersecting set of trees.
\end{conj}

\noindent \textbf{Subsequent work:} Conjecture~\ref{con:AK} has since been proven by Elizaveta Iarovikova and Andrey Kupavskii~\cite{AKtrees}. Conjecture~\ref{con:weak} was proven independently by Pitchayut Saengrungkongka~\cite{mark}.

\subsection{Other Directions}

We believe that other ambient graph classes, beyond the complete graphs, should admit similar results. Our expectation is that symmetric graphs and sufficiently dense graphs
would be good candidates, so
we expect that complete multipartite graphs and quasi-random graphs will work. Indeed, Lemma~\ref{lem:treeswithforest} and other enumerative properties of spanning trees of $K_n$ have multipartite analogues, see~\cite{multipartiteST,BonaW25}, for example.

We can generalize the spanning tree $t$-disjointness graph $\Gamma_t(K_n)$ to $\Gamma_t(G)$, for any graph $G$: the graph in which the vertices are all spanning trees of $G$ and two trees are adjacent if and only if they have fewer than $t$ edges in common. This opens up a world of different problems. For example, for a fixed $t \geq 2$, is it NP-hard to compute the clique (or independence) number of $\omega(\Gamma_t(G))$ for arbitrary $G$?  Can we find bounds or exact values of either $\omega(\Gamma_t(G))$ or $\alpha(\Gamma_t(G))$ for interesting families of graphs $G$? In fact, the clique number of $\omega(\Gamma_t(K_n))$ is open. 
\bigskip

\paragraph*{\bf{Acknowledgements}}
We thank Pitchayut Saengrungkongka, Tianhen Wang and Xiangxiang Zheng
for pointing out a mistake in an earlier version of this document.
The third author thanks Gil Kalai for suggesting the problem many years ago.

\bibliographystyle{plain}
\bibliography{ref}    
                    
\end{document}